\documentclass[12pt]{amsart} 
\usepackage{amsmath,amssymb,amscd,amsthm}
\usepackage{latexsym}
\usepackage{graphicx}
\usepackage[english]{babel}
\usepackage[latin1]{inputenc}       
\usepackage{pgf,pgfarrows,pgfnodes,pgfautomata,pgfheaps}
\usepackage{colortbl}

\usepackage[a4paper,twoside,left=3cm,right=2.8cm,top=3.1cm,bottom=2.3cm]{geometry}

\newtheorem{cor}{Corollary}
\newtheorem{Claim}{Claim}
\newtheorem{theorem}{Theorem}[section]
\newtheorem{proposition}[theorem]{Proposition}
\newtheorem{lemma}[theorem]{Lemma}

\newtheorem{remark}[theorem]{Remark}

\def\irr#1{{\rm Irr}(#1)}
\def\irrr#1#2 {\irr {#1 \mid #2}}
\newcommand{\R}{\mathbb R}

\newcommand{\sfe}{{{\mathbb S}^{n-1}}}

\newcommand{\E}{\mathbb E}

\begin{document}

\title[On the local version of the Log-Brunn-Minkowski conjecture]{On the local version of the Log-Brunn-Minkowski conjecture and some new related geometric inequalities}
\author[Alexander V. Kolesnikov, Galyna V. Livshyts]{Alexander V. Kolesnikov, Galyna V. Livshyts}
\address{National Research University Higher School of Economics, Russian Federation}
\email{sascha77@mail.ru}
\address{School of Mathematics, Georgia Institute of Technology, Atlanta, GA} \email{glivshyts6@math.gatech.edu}

\subjclass[2010]{Primary: 52} 
\date{\today}
\begin{abstract} We prove that for any semi-norm $\|\cdot\|$ on $\R^n,$ and any symmetric convex body $K$ in $\R^n,$
\begin{equation}\label{ineq-abs2}
\int_{\partial K} \frac{\|n_x\|^2}{\langle x,n_x\rangle}\leq \frac{1}{|K|}\left(\int_{\partial K} \|n_x\| \right)^2,
\end{equation}
and characterize the equality cases of this new inequality. The above would also follow from the Log-Brunn-Minkowski conjecture, if the latter was proven, and we believe that it may be of independent interest. We, furthermore, obtain an improvement of this inequality in some cases, involving the Poincare constant of $K.$ 

The conjectured Log-Brunn-Minkowski inequality is a strengthening of the Brunn-Minkowski inequality in the partial case of symmetric convex bodies, equivalent to the validity of the following statement: for all symmetric convex smooth sets $K$ in $\R^n$ and all smooth even $f:\partial K\rightarrow \R,$
\begin{equation}\label{ineq-abs}
\int_{\partial K} H_x f^2-\langle \mbox{\rm{II}}^{-1}\nabla_{\partial K}  f,  \nabla_{\partial K}  f\rangle +\frac{f^2}{\langle x,n_x\rangle}\leq \frac{1}{|K|}\left(\int_{\partial K} f   \right)^2.
\end{equation}
In this note, we verify (\ref{ineq-abs}) with the particular choice of speed function $f(x)=|\langle v,n_x\rangle|$, for all symmetric convex bodies $K$, where $v\in\R^n$ is an arbitrary vector.

\end{abstract}
\maketitle

\section{Introduction}

The Brunn-Minkowski inequality, proved in the full generality by Lusternik \cite{Lust}, states that:
\begin{equation}\label{BM}
|\lambda K+(1-\lambda)L|\geq |K|^{\lambda}|L|^{1-\lambda},
\end{equation}
which holds for all Borel-measurable sets $K, L$ and any $\lambda\in [0,1].$ Furthermore, due to the $n-$homogeneity of the Lebesgue measure, (\ref{BM}) self-improves to an a-priori stronger form
\begin{equation}\label{BM-add}
|\lambda K+(1-\lambda)L|^{\frac{1}{n}}\geq \lambda |K|^{\frac{1}{n}}+(1-\lambda)|L|^{\frac{1}{n}}.
\end{equation}
See a survey by Gardner \cite{Gar} on the subject for more information. 

B\"or\"oczky, Lutwak, Yang, Zhang \cite{BLYZ} conjectured that a stronger inequality, called the Log-Brunn-Minkowski inequality, holds in the case when $K$ and $L$ are symmetric convex sets:
\begin{equation}\label{logbm-leb}
|\lambda K+_0(1-\lambda)L|\geq |K|^{\lambda}|L|^{1-\lambda},
\end{equation}
where the zero-sum stands for 
$$\lambda K+_0(1-\lambda)L:=\{x\in\R^n:\,\forall\,u\in\sfe\,|\langle x,u\rangle|\leq h_K(u)^{\lambda}h_L(u)^{1-\lambda}\};$$
here the support function of a convex set $K$ is defined to be
$$h_K(x):=\sup_{y\in K}\langle x,y\rangle.$$

B\"or\"oczky, Lutwak, Yang, Zhang \cite{BLYZ} verified this conjecture for planar symmetric convex sets. Saroglou \cite{christos1} and Cordero-Erasquin, Fradelizi and Maurey \cite{CFM} proved the conjecture for unconditional convex sets in $\R^n$; recently, a stronger result, with a weaker symmetry assumption, was derived by B\"or\"oczky, Kalantzopoulos \cite{BorKal}. Rotem \cite{liran} verified the conjecture for complex convex bodies. 

Colesanti, Livshyts, Marsiglietti \cite{CLM} (Theorem 6.5), and later Kolesnikov, Milman \cite{KolMilsupernew} derived the local version of the Log-Brunn-Minkowski conjecture: in case the Log-Brunn-Minkowski conjecture holds, then for every symmetric strictly-smooth convex body $K$ and every smooth even function $f:\partial K\rightarrow\R,$
\begin{equation}\label{LogBM-infin}
\int_{\partial K} H_x f^2-\langle \mbox{\rm{II}}^{-1}\nabla_{\partial K}  f,  \nabla_{\partial K}  f\rangle +\frac{f^2}{\langle x,n_x\rangle}\leq \frac{1}{|K|}\left(\int_{\partial K} f   \right)^2.
\end{equation}
Here $\rm{II}$ is the second fundamental form of $K$, the mean curvature $H_x=tr(\rm{II})$, and $n_x$ stands for the normal unit vector to $K$ at the point $x$. The integration runs in the Hausdorff $(n-1)-$dimensional measure on the boundary of $K.$ Chen, Huang, Li, Liu \cite{global} and Putterman \cite{Put} showed that the validity of (\ref{LogBM-infin}) for any $C^2$-smooth symmetric convex set $K,$ and any $C^1$-smooth even function $f:\partial K\rightarrow \R$, yields back the Log-Brunn-Minkowski conjecture.

Kolesnikov and Milman \cite{KolMilsupernew} noticed that for any $t\in\R,$ the inequality (\ref{LogBM-infin}) is invariant under the transformation $f\rightarrow f+t\langle x,n_x\rangle,$ or, in other words, the Log-Brunn-Minkowski inequality is invariant under the transformation $L\rightarrow sL$. For the reader's convenience, we outline it as a 

\begin{Claim}\label{claim-invar}
The inequality (\ref{LogBM-infin}) is invariant under the transformation $f\rightarrow f+t\langle x,n_x\rangle.$ 
\end{Claim}
\begin{proof}
Indeed, consider the bilinear form
$$Q(f,g)= \int_{\partial K} H_x fg-\langle \mbox{\rm{II}}^{-1}\nabla_{\partial K}  f,  \nabla_{\partial K}  g \rangle+\frac{fg}{\langle x,n_x\rangle}-\frac{1}{|K|}\left(\int_{\partial K} f   \right) \left(\int_{\partial K} g   \right).$$
Integrating by parts inside the boundary (see Kolesnikov, Milman \cite{KM1}, subsection 2.1, for more details), and using the fact that $\nabla^2 \frac{x^2}{2}=Id,$ we get
$$
\int_{\partial K} H_x f\langle x,n_x\rangle-\langle \mbox{\rm{II}}^{-1}\nabla_{\partial K}  f,  \nabla_{\partial K}  \langle x,n_x\rangle\rangle= (n-1)\int_{\partial K} f.
$$
Therefore, using the fact that $\int_{\partial K} \langle x,n_x\rangle=n|K|,$ we see that for any smooth even function $f:\partial K\rightarrow\R,$
$$Q(f,\langle x,n_x\rangle)=0.$$
By bilinearlity, we thus get 
$$Q(f+t\langle x,n_x\rangle, f+t\langle x,n_x\rangle)=Q(f,f).$$
\end{proof}

Colesanti, Livshyts, Marsiglietti \cite{CLM} observed that (\ref{LogBM-infin}) holds for any even smooth function $f:\partial K\rightarrow \R$ in the partial case when $K$ is the Euclidean ball centered at the origin. Indeed, in this case (\ref{LogBM-infin}) boils down to showing that for any smooth  \emph{even} function $\varphi:\sfe\rightarrow\R,$
$$Var(\varphi)\leq \frac{1}{n}\E|\nabla_{\sigma} \varphi|^2,$$
where $\nabla_{\sigma}$ is the spherical gradient, and the variance and the expectation are taken with respect to the random vector distributed uniformly on the sphere. This fact is true, since for even functions the above holds, moreover, with the constant $\frac{1}{2n}<\frac{1}{n},$ in view of the fact that the second eigenvalue of the Laplacian on the unit sphere is ${2n}$. Kolesnikov and Milman \cite{KolMilsupernew} showed, furthermore, that (\ref{LogBM-infin}) holds for any even smooth function $f:\partial K\rightarrow \R$ in the case when $K=B_p^n$, for any $p\in [2,\infty],$ for sufficiently large dimension $n\geq N(p).$

\medskip

Recall the notion of mixed volumes for symmetric convex bodies $K$ and $M$: for $k=1,...,n,$
$$V_k(K,M)=\frac{(n-k)!}{n!}|K+tM|^{(k)}_{t=0},$$
where $\cdot^{(k)}$ stands for the $k$-th derivative. Recall that Minkowski's first inequality states:
$$V_1(K,M)\geq |K|^{\frac{n-1}{n}}|L|^{\frac{1}{n}}.$$ 

In this terminology, considering a symmetric convex set $M$ and letting $f(x)=h_M(n_x)$, we rewrite the inequality (\ref{LogBM-infin}):
\begin{equation}\label{LogBM-norms}
n(n-1)V_{2}(K,M)+\int_{\partial K} \frac{h_M^2(n_x)}{\langle x,n_x\rangle}\leq \frac{n^2V_{1}(K,M)^2}{|K|}.
\end{equation}
In view of H\"older's inequality, we have
$$\int_{\partial K} \frac{h_M^2(n_x)}{\langle x,n_x\rangle}\geq \frac{nV_{1}(K,M)^2}{|K|},$$
and thus (\ref{LogBM-norms}) is a direct strengthening, for symmetric sets, of Minkowski's Second inequality (a partial case of Minkowski's quadratic inequality):
\begin{equation}\label{MinkQ}
V_{2}(K,M)\leq \frac{V_{1}(K,M)^2}{|K|}.
\end{equation}

Putterman \cite{Put} came up with the idea to consider the formulation (\ref{LogBM-norms}), and explained (in the appendix) the equivalence of (\ref{LogBM-norms}) and (\ref{LogBM-infin}). For completeness, we outline this equivalence below.

\begin{Claim}\label{claim-equiv}
The validity of (\ref{LogBM-infin}) for any symmetric strictly-convex $C^2-$smooth body $K$ in $\R^n$ and any even $f\in C^1(\partial K)$ is equivalent to the validity of (\ref{LogBM-norms}) for any symmetric convex body $K$ and any symmetric bounded convex set $M$ in $\R^n.$  
\end{Claim}
\begin{proof} First, recall that for any symmetric strictly-convex $C^2-$smooth body $K$ and any symmetric bounded $C^2$-smooth convex set $M$ in $\R^n$, we have, as per Kolesnikov, Milman \cite{KolMil-1}, Colesanti \cite{Col1}, or Bonnesen, Fenchel \cite{BF}:
$$nV_1(K,M)=|K+tM|'_{t=0}=\int_{\partial K} h_M(n_x),$$
and
$$n(n-1)V_2(K,M)=|K+tM|''_{t=0}=	\int_{\partial K} H_x h_M(n_x)^2-\langle \mbox{\rm{II}}^{-1}\nabla_{\partial K}  h_M(n_x),  \nabla_{\partial K}  h_M(n_x)\rangle.$$  
Therefore, the validity of (\ref{LogBM-infin}) for any symmetric strictly-convex $C^2-$smooth body $K$ in $\R^n$ and any even $f\in C^1(\partial K)$ implies the validity of (\ref{LogBM-norms}) for any symmetric strictly-convex $C^2-$smooth body $K$ and any symmetric $C^2-$smooth bounded convex set $M$ in $\R^n.$  

Next, the class of symmetric $C^2$-smooth strictly convex bodies is dense in the Hausdorff metric in the class of all symmetric convex bodies in $\R^n,$ and the class of symmetric $C^2-$smooth bounded convex sets is dense in the class of symmetric bounded convex sets. Recall that mixed volumes are continuous in the Hausdorff metric in every variable (see, e.g. Artstein-Avidan, Giannopolous, Milman \cite{AGM} Theorem B.1.14, or the books by Schneider \cite{book4}, or Bonnesen, Fenchel \cite{BF}). Therefore, the validity of (\ref{LogBM-infin}) in the smooth case implies the validity of (\ref{LogBM-norms}) for any symmetric convex body $K$ and any symmetric bounded convex set $M$ in $\R^n.$  

Lastly, to get the converse implication, notice that for any symmetric strictly-convex $C^2-$smooth body $K$ in $\R^n$, any function $f\in C^1(\partial K)$ can be written in the form $f(x)=h_M(n_x)-th_K(n_x),$ for some $t\in\R,$ and some symmetric convex set $M.$ By Claim \ref{claim-invar}, the validity of (\ref{LogBM-infin}) for $K$ and $f(x)$ is equivalent to the validity of (\ref{LogBM-infin}) for $K$ and $h_M(n_x).$ In other words, it suffices to verify (\ref{LogBM-infin}) only for functions $f$ given by $f(x)=h_M(n_x)$, and this follows from the validity of (\ref{LogBM-norms}).
\end{proof}

\begin{remark} The inequality (\ref{LogBM-norms}) holds for any pair of $K$ and $M$ such that $h_M(u)=h_K(u)$ for all $u$ in the support of $dS_K$  (the surface area measure of $K$.) Indeed, in this case, noting that $h_M(n_x)=h_K(n_x)=\langle x,n_x\rangle,$ we write
$$\int_{\partial K} \frac{h_M^2(n_x)}{\langle x,n_x\rangle}dH_{n-1}(x)=n|K|,$$
and
$$\frac{n^2V_{1}(K,M)^2}{|K|}=n^2|K|.$$
Therefore, 
$$\int_{\partial K} \frac{h_M^2(n_x)}{\langle x,n_x\rangle}=\frac{1}{n}\frac{n^2V_{1}(K,M)^2}{|K|},$$
and (\ref{LogBM-norms}) follows from Minkowski's Second inequality (\ref{MinkQ}). Therefore, (\ref{LogBM-norms}) holds, for example, when $K$ is a polytope circumscribed around the unit ball, and $M=B_2^n.$

We note also that all the equality cases of Minkowski's Second inequality (characterized by Bol \cite{Bol}, and in greater generality by Shenfeld and van Handel \cite{vHS}) are also equality cases of the Log-Brunn-Minkowski inequality (though there is potentially more equality cases of the Log-Brunn-Minkowski inequality itself). Indeed, Theorem 3.4 from \cite{vHS} implies that if symmetric convex sets $M$ and $K$ give an equality in Minkowski's second inequality, and $K$ is full-dimensional, then (in particular) there exists an $\alpha>0$ such that for any $\theta\in\sfe\cap supp(dS_K),$ one has $h_M(\theta)=\alpha h_K(\theta).$ This, in turn, implies that
\begin{equation}\label{vHSref}
\int_{\partial K} \frac{h_M^2(n_x)}{\langle x,n_x\rangle}=n\frac{V_1(K,M)^2}{|K|}.
\end{equation}
Thus if $K$ and $M$ give an equality in Minkowski's Second inequality (\ref{MinkQ}), then the equality also holds in (\ref{LogBM-norms}), and thus the equality also holds in the Log-Brunn-Minkowski inequality (in view of the invariance property of (\ref{LogBM-infin}).)

Note that if it wasn't for (\ref{vHSref}) for all of the equality cases of MSI, then one would have a counterexample to the Log-Brunn-Minkowski conjecture, as the MSI is weaker than (\ref{LogBM-norms}); the work of Bol \cite{Bol} and Shenfeld and van Handel \cite{vHS}) prevents it, however.
\end{remark}

The following remark was independently noticed by Emanuel Milman, of which the authors were not aware initially.
\begin{remark}
As mentioned earlier, Kolesnikov and Milman \cite{KolMilsupernew} verified (\ref{LogBM-infin}) for $K=B^n_{\infty}$. We note that the reduction of (\ref{LogBM-infin}) to (\ref{LogBM-norms}) allows to check this fact in a rather elementary manner: letting $K=B^n_{\infty}$ and $\varphi=h_M,$ for some symmetric convex bounded set $M,$ the inequality (\ref{LogBM-norms}) rewrites as
\begin{equation}\label{rem1}
n(n-1)V_{2}(B^{n}_{\infty},M)+2\cdot 2^{n-1}\sum_{i=1}^n  \varphi^2(e_i)\leq 2^{n}\left(\sum_{i=1}^n  \varphi(e_i)\right)^2,
\end{equation}
where we used the fact that norms are even functions, and thus $\varphi(e_i)=\varphi(-e_i),$ for all $i=1,...,n.$ Here $e_1,...,e_n$ denotes the canonical basis.

Let $B_{\varphi}$ be the origin-centered coordinate parallelepiped with sides $2 \varphi(e_1),...,2 \varphi(e_n),$ and note that $M\subset B_{\varphi}.$ Recall that mixed volumes are monotone, and therefore,
\begin{equation}\label{rem2} 
V_{2}(B^{n}_{\infty},M)\leq V_{2}(B^{n}_{\infty},B_{ \varphi}).
\end{equation}
By differentiating the determinant of the matrix with diagonal elements $1+t \varphi(e_1),...,1+t \varphi(e_n)$, we see that
\begin{equation}\label{rem3} 
n(n-1)V_{2}(B^{n}_{\infty},B_{ \varphi})=2\cdot2^{n}\sum_{i\neq j}  \varphi(e_i) \varphi(e_j).
\end{equation}
We conclude by noticing that the inequality (\ref{rem1}) follows from (\ref{rem2}), (\ref{rem3}), and the equality
$$\sum_{i=1}^n \sum_{j=1}^n \varphi(e_i)\varphi(e_j)=\left(\sum_{i=1}^n  \varphi(e_i)\right)^2.$$
In conclusion, (\ref{LogBM-norms}) holds when $K=B^n_{\infty}$.
\end{remark}

\medskip




In view of the fact that mixed volumes are non-negative, the following is true:

\medskip
\medskip

\emph{If the Log-Brunn-Minkowski conjecture holds, then for any symmetric convex body $K\subset\R^n$ and any semi-norm $\|\cdot\|$ on $\R^n,$ one has
\begin{equation}\label{ourguy}
\int_{\partial K} \frac{\|n_x\|^2}{\langle x,n_x\rangle}\leq \frac{1}{|K|}\left(\int_{\partial K} \|n_x\|  \right)^2.
\end{equation}}

\medskip
\medskip

Indeed, letting $f(x)=\|n_x\|_{M^o}$ in (\ref{LogBM-infin}), we obtain (\ref{LogBM-norms}), which, in turn, yields (\ref{ourguy}), in view of the fact that $V_2(K,M)\geq 0.$ As the Log-Brunn-Minkowski inequality is not known in general, (\ref{ourguy}) is not known a-priori. In this note, we show that (\ref{ourguy}) is indeed true:

\begin{theorem}\label{thm-zono}
For any symmetric convex bounded set $K$ in $\R^n$ with non-empty interior, and any semi-norm $\|\cdot\|$ on $\R^n$, we have
$$\int_{\partial K} \frac{\|n_x\|^2}{\langle x,n_x\rangle}\leq \frac{1}{|K|}\left(\int_{\partial K} \|n_x\| \right)^2.$$
Furthermore, the equality occurs if and only if $\|\cdot\|=|\langle \cdot,v\rangle|$, for some vector $v\in\R^n,$ and $K=C+[-v,v]$ for some $(n-1)-$dimensional symmetric convex set $C$.
\end{theorem}


Note that the inequality 
\begin{equation}\label{false}
\int_{\partial K} \frac{f^2}{\langle x,n_x\rangle}\leq \frac{1}{|K|}\left(\int_{\partial K} f \right)^2.
\end{equation}
cannot hold without the strong assumption of $f$ being a semi-norm of $n_x$: indeed, if $f(x)=1_{\Omega}$ for a very small set $\Omega\subset \partial K$, then (\ref{false}) should fail, as one might note. Another example is $K=B^n_1$: in this case, (\ref{false}) boils down to
$$\sum_{i=1}^{2^n} f^2(u_i)\leq\frac{n}{2^n} \left(\sum_{i=1}^{2^n} f(u_i)\right)^2,$$
where $u_i$ are the unit normals to the faces of $B_1^n.$ When $f(u_i)=f(-u_i)=1$ for a fixed index $i$, and $f(u_j)=0$ for all other indices $j\neq i$, the above becomes
$$2\leq \frac{4n}{2^n},$$
which fails for any $n>2.$

\medskip
\medskip

By $r(K)$ denote the in-radius of $K$ (the radius of the largest euclidean ball contained in $K$), and $C_{poin}(K)$ denotes the Poincare constant of $K$. That is, $C_{poin}(K)$ is the smallest number such that
\begin{equation}\label{PCdef}
\int_K g^2-\frac{1}{|K|}\left(\int_K g\right)^2\leq C^2_{poin}(K)\int_K |\nabla g|^2,
\end{equation}
for all $C^1-$smooth functions $g$ on $K,$ such that the integrals above exist. To complement Theorem \ref{thm-zono}, we obtain another estimate, which may be stronger in some cases:

\begin{theorem}\label{thm-bound}
For any symmetric convex body $K$ and any semi-norm $\|\cdot\|$, we have
$$
\int_{\partial K} \frac{\|n_x\|^2}{\langle x,n_x\rangle}\leq \inf_{T}\frac{2C_{poin}(TK)}{r(TK)}\cdot \frac{\left(\int_{\partial K} \|n_x\| \right)^2}{|K|},
$$
where the infimum runs over all volume preserving linear operators $T$.
\end{theorem}

\begin{remark}
Recall that $\inf_{T}\frac{2C_{poin}(TK)}{r(TK)}\leq Cn^{1/4},$ as follows from the work of Lee, Vempala \cite{LV} and Kannan, Lovasz, and Simonovits \cite{KLS}. 

Note that for some convex bodies $K,$ the expression $\inf_{T}\frac{2C_{poin}(TK)}{r(TK)}$ could be significantly smaller than $1$, and therefore in some cases, Theorem \ref{thm-bound} is stronger than Theorem \ref{thm-zono}. For example, when $K$ is an ellipsoid, 
$$\inf_{T}\frac{2C_{poin}(TK)}{r(TK)}=\frac{C}{\sqrt{n}},$$
and we have, for every semi-norm $\|\cdot\|$,
$$
\int_{\partial K} \frac{\|n_x\|^2}{\langle x,n_x\rangle}\leq \frac{c}{\sqrt{n}} \frac{\left(\int_{\partial K} \|n_x\| \right)^2}{|K|}.
$$
More generally, when $K$ is a linear image of an $L_p$-ball, we get the inequality

$$
\int_{\partial K} \frac{\|n_x\|^2}{\langle x,n_x\rangle}\leq cn^{-
\min(\frac{1}{2},\frac{1}{p})} \frac{\left(\int_{\partial K} \|n_x\| \right)^2}{|K|}.
$$

\end{remark}

As an immediate corollary of Theorems \ref{thm-zono} and \ref{thm-bound} and Minkowski's Second inequality (\ref{MinkQ}), we get

\begin{cor}\label{cor}
For every pair of symmetric convex bodies $K$ and $M$,
$$n(n-1)V_{2}(K,M)+\int_{\partial K} \frac{h_M^2(n_x)}{\langle x,n_x\rangle}\leq $$$$\left(\frac{n-1}{n}+\min\left(1,\frac{2C_{poin}(K)}{r(K)}\right)\right)\frac{n^2V_{1}(K,M)^2}{|K|}.$$
\end{cor}
In particular, the inequality (\ref{LogBM-norms}) holds with a factor of 2 (and even with a factor of $\frac{2n-1}{n}$). 

\begin{remark} Corollary \ref{cor} implies that the function $|K+_0 tM|^{-\frac{n-1}{n}}$ is convex at $t=0$. Indeed, letting $F(t)=|K+_0 tM|$, we see (from \cite{CLM}, \cite{KolMilsupernew} or \cite{HKL}) that
$$(F^{-\frac{n-1}{n}})''(0)\geq 0$$ 
is equivalent to the inequality
$$n(n-1)V_{2}(K,M)+\int_{\partial K} \frac{h_M^2(n_x)}{\langle x,n_x\rangle}\leq \frac{2n-1}{n}\frac{n^2V_{1}(K,M)^2}{|K|},$$
which follows from Corollary \ref{cor}. Unfortunately, this inequality is not invariant under the change $h_M(n_x)\rightarrow h_M(n_x)+\langle x,n_x\rangle.$ If it was, then using the local-to-global argument of Putterman \cite{Put}, we would get that for any pair of symmetric convex bodies $K$ and $L$ and any $\lambda\in [0,1]$, one has
$$|\lambda K+_0(1-\lambda)L|^{-\frac{n-1}{n}}\leq \lambda |K|^{-\frac{n-1}{n}}+(1-\lambda)|L|^{-\frac{n-1}{n}}.$$
\end{remark}

\medskip

In addition, we show the following result, which is a corollary of a result of Giannopoulos and Hartzoulaki \cite{GH}.

\begin{proposition}\label{prop}
For any pair of symmetric convex bodies $K$ and $M$ with non-empty interior, there exists a linear operator $T$, depending on $K$ and $M,$ such that  
$$n(n-1)V_{2}(K,TM)+\int_{\partial K} \frac{h_{TM}^2(n_x)}{\langle x,n_x\rangle}\leq \left(1+\frac{C\log n}{\sqrt{n}}\right)\frac{n^2V_{1}(K,TM)^2}{|K|},$$
where $C>0$ is an absolute constant.
\end{proposition}

\medskip
\medskip

In contrast with the existing results about the Log-Brunn-Minkowski inequality, it is interesting to investigate the validity of (\ref{LogBM-infin}) \emph{for all symmetric convex sets $K$}, given that $f(x)=\varphi(n_x)$ on $\partial K,$ for some fixed function $\varphi:\sfe\rightarrow\R.$ Prior to this note, such a result was not known for any non-trivial examples of $\varphi$. To this end, we show the following
\begin{theorem}\label{thm-int}
For every symmetric strictly-convex $C^2$-smooth bounded set $K$ in $\R^n$ with non-empty interior, for $f(x)=t\langle x,n_x\rangle+|\langle v,n_x\rangle|$, for any $t\in\R$ and any $v\in\R^n,$ the inequality (\ref{LogBM-infin}) is true. In other words, (in view of the invariance of (\ref{LogBM-infin}) under the change $f\rightarrow f+t\langle x,n_x\rangle$),
$$\int_{\partial K} H_x \langle n_x,v\rangle^2-\langle \mbox{\rm{II}}^{-1}\nabla_{\partial K}  |\langle n_x,v\rangle|,  \nabla_{\partial K}  |\langle n_x,v\rangle|\rangle +\frac{\langle n_x,v\rangle^2}{\langle x,n_x\rangle}\leq \frac{1}{|K|}\left(\int_{\partial K} |\langle n_x,v\rangle|  \right)^2.$$

Equivalently, for every symmetric bounded convex set $K$ in $\R^n$ with non-empty interior, letting $M=[-v,v]$, we have
\begin{equation}\label{ineq-ref}
n(n-1)V_{2}(K,M)+\int_{\partial K} \frac{h_M^2(n_x)}{\langle x,n_x\rangle}\leq \frac{n^2V_{1}(K,M)^2}{|K|}.
\end{equation}
Furthermore, the equality in (\ref{ineq-ref}) is attained if and only if $K=C+[-v,v]$ for some $(n-1)-$dimensional symmetric convex $C$.
\end{theorem}

\begin{remark}
The equivalence of the two formulations of the Theorem \ref{thm-int} was outlined when we proved Claim \ref{claim-equiv}.
\end{remark}

Speaking in vague terms, Theorem \ref{thm-int} indicates that the Local Log-Brunn-Minkowski inequality ``holds in one direction''; it is, in some sense, ``a localized version'' of this inequality. It is tempting to try and derive the conjecture from such localized version, however it will be clear in Section 6, that in order to hope to obtain any general result, one must deduce a ``localized'' version in at least two directions, which we believe to be harder. 

We would like to emphasize that Theorem \ref{thm-int} has no relation whatsoever to the local version of the global inequality
$$|\lambda K+_0(1-\lambda)[-v,v]|_n\geq |K|_n^{\lambda}|[-v,v]|^{1-\lambda}_n,$$
which is in itself trivial, since both sides if the inequality are zero.

In Section 2 we discuss some preliminaries. In Section 3 we prove Theorems \ref{thm-int} and \ref{thm-zono}. In Section 4 we prove Theorem \ref{thm-bound}. In Section 5 we prove Proposition \ref{prop}. In Section 6 we discuss interesting open questions and reductions. 

\textbf{Acknowledgement.} The authors are grateful to Yair Shenfeld for some illuminating discussions on the subject of the equality cases in Minkowski's second inequality. The first named author was supported by RFBR project 20-01-00432. The article was prepared within the framework of the HSE University Basic Research Program. The second named author is supported by the NSF CAREER DMS-1753260. 

\section{Preliminaries}

A convex body is a compact convex set with non-empty interior. Throughout the paper, $K, L, M$ stand for convex bodies in $\R^n$. Lebesgue volume shall be denoted by $|\cdot|$ as well as $|\cdot|_n$, and $|\cdot|_k$ stands for the $k$-dimensional Lebesgue measure, for $k\leq n.$ For a point $x\in\partial K,$ we denote by $n_x$ the unit outer normal vector to the boundary of $K$ at $x$; it is uniquely defined at any regular boundary point $x$ of $K$, and almost every boundary point of $K$ is regular (see Schneider \cite{book4}). The Gauss map is the map which acts from $\partial K$ into $\sfe$ and associates to any point $x$ its unit normal(s) $n_x$. Recall the notation $dS_K$ for the surface area measure of $K$ on the sphere, the push forward of the boundary measure $dH_{n-1}$ of $K$ onto the sphere under the Gauss map. For brevity, we will use the notation $\int_{\partial K} g$ in place of $\int_{\partial K} g(x) dH_{n-1}(x)$.

We say that a convex body $K$ is $C^2-$smooth if its boundary is a surface of class $C^2.$ We say that $K$ is in the class $C^{2,+}$ if it is $C^2$ smooth and the Gauss curvature is strictly positive for all the boundary points.

When $K$ is a polytope with normals $u_i$ and areas of facets $F_i$, with $i=1,...,N,$ we have
$$dS_K=\sum_{i=1}^N F_i \delta_{u_i},$$
where $\delta_{u_i}$ stands for the delta-measure at $u_i.$ Letting $h_i=h_K(u_i)$, we see that the inequality from Theorem \ref{thm-zono} for polytopes rewrites as
$$\sum_{i=1}^N \frac{F_i}{h_i}\|u_i\|^2\leq \frac{\left(\sum_{i=1}^N F_i\|u_i\|\right)^2}{|K|},$$
where $\|\cdot\|$ is an arbitrary semi-norm. In particular, when $K$ is a polytope circumscribed around $B_2^n,$ this boils down to
$$\E\|X\|^2\leq n\left(\E\|X\|\right)^2,$$
where $X$ is a random vector uniformly distributed over the finite set $supp(S_K)\subset\sfe$. 

\begin{Claim}{[A part of Minkowski's Existence and Uniqueness theorem]}\label{claim}
When a convex bounded set $K$ in $\R^n$ has non-empty interior, the span of the support of the measure $dS_K$ is the entire $\R^n.$ 
\end{Claim}
This is one of the aspects of Minkowski's theorem, see, e.g. Schneider \cite{book4}.

For $u\in\R^n$, we will use the notation
$$u^{\perp}=\{x\in\R^n:\,\langle x,u\rangle=0\}$$
for the hyperplane orthogonal to $u.$ More generally, for a linear space $H$ we denote its orthogonal complement by $H^{\perp}$. Hyperplane sections of a convex set $K$ will be denoted by $K\cap u^{\perp}.$ We shall also consider orthogonal projections of convex sets onto linear subspaces $H$, given by
$$K|H=\{x\in H:\,\exists\,\, y\in H^{\perp}:\,\,x+y\in K\}.$$

Recall the Cauchy's projection formula (see, e.g. Koldobsky \cite{Kold}), which states that for any unit vector $v$,
\begin{equation}\label{Cauchy-formula}
|K|v^{\perp}|_{n-1}=\frac{1}{2}\int_{\partial K} |\langle n_x, v\rangle| =\frac{1}{2}\int_{\sfe} |\langle \theta, v\rangle| dS_K(\theta).
\end{equation}

The Minkowski functional of a symmetric convex bounded set $K$ with non-empty interior is the norm
$$\|x\|_K=\inf\{ t>0:\, x\in tK\}.$$
In case $K$ has empty interior, its Minkowski functional is a semi-norm. The polar body of $K$ is the convex body
$$K^o=\{x\in\R^n:\,\forall y\in K,\,\langle x,y\rangle\leq 1\}.$$
Recall that $K=K^{oo}$ in $\R^n,$ and that $h_{K^o}(x)=\|x\|_{K}$. An important property of support functions is the fact that
$$h_{K+L}=h_K+h_L.$$

When $K$ is an interval $[-v,v]$, for a vector $v\in\R^n,$ then
$$h_{[-v,v]}(u)=|\langle u,v\rangle|.$$
A convex body $M$ is called a zonotope if it is a Minkowski sum of intervals, or in other words,
$$h_M(x)=\sum_{i=1}^N \alpha_i |\langle x, v_i\rangle|,$$
for a collection of unit vectors $v_i$ and non-negative numbers $\alpha_i$. A cube $B^n_{\infty}$ is a zonotope, while a cross-polytope $B^n_1$ is not a zonotope. A convex body is called a zonoid, in case it is a limit of zonotopes, or, in other words
$$h_M(x)=\int_{\sfe} |\langle x, v\rangle| d\mu(v),$$
for some measure $\mu$ on $\sfe.$ For $p\in [2,\infty],$ the set $B_p^n$ is a zonoid. See, e.g., Koldobsky, Ryabogin, Zvavitch \cite{KRZ} for more information about zonoids.

Pick vectors $v, w\in\R^n$ with $\langle v,w\rangle\neq 0$, and pick a convex bounded set $C\subset w^{\perp}.$ A (tilted) symmetric \emph{cylinder} with \emph{axes} $v$ and \emph{base} $C$ is a convex body given by
$$K:=\{x+tv,\,x\in C, t\in [-1,1]\}=C+[-v,v].$$

Recall also that the first mixed volume has the integral representation:
$$V_1(K,L)=\frac{1}{n}\int_{\sfe} h_L(u)dS_K(u)=\frac{1}{n}\int_{\partial K} h_L(n_x).$$



\section{Proof of Theorems \ref{thm-int} and \ref{thm-zono}.} 


Throughout the section, $K$ denotes a symmetric compact convex set with non-empty interior. We shall use a few times the following simple fact:

\begin{lemma}\label{h_K}
For any $u\in\sfe$, 
$$\frac{1}{h_K(u)}\leq \frac{2|K\cap u^{\perp}|}{|K|}.$$
The equality occurs if and only if $K$ is a cylinder with the base orthogonal to $u.$
\end{lemma}
\begin{proof} By Fubini's theorem, for every $u\in\sfe,$
$$|K|=\int_{-h_K(u)}^{h_K(u)}|K\cap (u^{\perp}+tu)|dt\leq 2h_K(u)|K\cap u^{\perp}|,$$
where the last inequality follows since the maximal section of a symmetric convex set is the central one (which follows, e.g., from the Brunn-Minkowski inequality; see, e.g. Koldobsky \cite{Kold}). The equality occurs if all the sections $|K\cap (u^{\perp}+tu)|$ have equal area, which occurs only if $K$ is a cylinder with the base orthogonal to $u.$
\end{proof}

We say that a collection $V$ of vectors is collinear (or parallel) to one another if for any pair $x,y\in V$ we have $|\langle x,y\rangle|=|x|\cdot |y|.$

As a consequence of Lemma \ref{h_K}, we get

\begin{lemma}\label{max}
For a symmetric bounded convex set $K$ with non-empty interior, and any set $\Omega\in\R^n,$ for every vector $u\in\sfe,$ one has
\begin{equation}\label{lemma1-sect3}
\frac{\sup_{v\in\Omega}|\langle u,v\rangle|}{h_K(u)}\leq \frac{\int_{\sfe} \sup_{v\in\Omega}|\langle \theta,v\rangle| dS_K(\theta)}{|K|}.
\end{equation}
The equality occurs if and only if $\Omega$ consists only of vectors collinear to one another and non-orthogonal to $u$, and $K$ is a cylinder with base orthogonal to $u$ and axes parallel to all the vectors in $\Omega.$
\end{lemma}
\begin{proof} First, note that for any unit vector $v,$ any $(n-1)-$dimensional affine hyperplane $H$ orthogonal to the unit vector $u$, and any measurable set $A\subset H$, we have
$$|A|v^{\perp}|_{n-1}=|A|_{n-1}\cdot |\langle u,v\rangle|.$$
Indeed, this is elementary for $n=2,$ and follows by Fubbini's theorem for $n\geq 2.$

Thus for any pair of vectors $v,u\in\sfe,$
\begin{equation}\label{subset}
|K\cap u^{\perp}|\cdot |\langle u,v\rangle|=|(K\cap u^{\perp})|v^{\perp}|\leq |K|v^{\perp}|,
\end{equation}
where in the last passage we observed that a projection of a subset is smaller than a projection of a set.

By Lemma \ref{h_K},
\begin{equation}\label{lastnum1}
\frac{\sup_{v\in\Omega}|\langle u,v\rangle|}{h_K(u)}\leq \frac{2}{|K|}|K\cap u^{\perp}|\cdot \sup_{v\in\Omega}|\langle u,v\rangle|\leq  \frac{2}{|K|}\sup_{v\in\Omega} |K|v^{\perp}|,
\end{equation}
where in the last passage we used (\ref{subset}).

On the other hand, by Cauchy's projection formula (\ref{Cauchy-formula}), and in view of the fact that the integral of the supremum of a non-negative quantity is larger than the supremum of integrals, we get
\begin{equation}\label{lastnum2}
2\sup_{v\in\Omega} |K|v^{\perp}|\leq \int_{\sfe} \sup_{v\in\Omega}|\langle \theta,v\rangle| dS_K(\theta).
\end{equation}
By scaling, it is enough to show (\ref{lemma1-sect3}) when $v\in\sfe.$ Thus the inequality (\ref{lemma1-sect3}) follows from (\ref{lastnum1}) and (\ref{lastnum2}).

Suppose now the equality holds in (\ref{lemma1-sect3}). Firstly, this means that the equality holds in the first passage of (\ref{lastnum1}), and thus by Lemma \ref{h_K}, $K$ must be a cylinder with base orthogonal to $u$. 

Let $\tilde{\Omega}\subset \Omega$ be the set of vectors $v$ in $\Omega$ which maximize $|\langle u,v\rangle|$. Note that for all $v\in \tilde{\Omega},$ we have $|\langle u,v\rangle|>0:$ indeed, otherwise, if the equality holds in (\ref{lemma1-sect3}), then all the vectors in the support of $dS_K(\theta)$ are orthogonal to all the vectors in $\Omega$. But, by Claim \ref{claim}, this contradicts the fact that $K$ has non-empty interior, and therefore the span of the support of $dS_K$ is the entire $\R^n.$

Therefore, for all $v\in \tilde{\Omega},$ we have $|\langle u,v\rangle|>0.$ Since the equality holds in (\ref{lemma1-sect3}), the equality also holds in (\ref{subset}), and therefore, $K$ (which is a cylinder with base orthogonal to $u$) has axes parallel to $v_0,$ some vector $v_0\in\tilde{\Omega}$. 

Next, in order for the equality
$$\sup_{\alpha\in A} \int f_{\alpha}(x)d\mu(x)=\int \sup_{\alpha\in A} f_{\alpha}(x)d\mu(x)$$
to occur, for some family of functions $f_{\alpha}\geq 0,$ indexed by some set $A,$ we must have $f_{\alpha}(x)=f_{\beta}(x)$, for almost every $x\in supp(\mu)$ and for all $\alpha,\beta\in A.$ Therefore, under the assumption of the equality in  (\ref{lemma1-sect3}), we have, for almost every $\theta$ from the support of $dS_K$, that $|\langle \theta,v_1\rangle|=|\langle \theta,v_2\rangle|$, for all $v_1,v_2\in \Omega$. But we have already concluded that $K$ is a cylinder with axes parallel to $v_0$ and base orthogonal to $u.$ Therefore, for all $\theta\in supp(S_K)\setminus \{u\},$ we have $\langle \theta,v_0\rangle=0.$ Thus $\Omega$ is contained in the set orthogonal to $supp(S_K)\setminus \{u\}.$ But since $K$ is a bounded convex set with non-empty interior, Claim \ref{claim} implies that the span of the support of $dS_K$ is the entire $\R^n,$ and hence the dimension of the set $supp(S_K)\setminus \{u\}$ is at least $n-1$. Therefore, all the vectors in $\Omega$ are parallel to $v_0.$ The proof is complete.
\end{proof}

As corollary, we notice

\begin{cor}\label{point}
For a symmetric convex bounded set $K$ with non-empty interior, and any vectors $u, v\in\R^n\setminus \{0\},$
\begin{equation}\label{lemma2-sect3}
\frac{|\langle u,v\rangle|}{h_K(u)}\leq \frac{\int_{\sfe} |\langle \theta,v\rangle| dS_K(\theta)}{|K|}.
\end{equation}
The equality holds if and only if $K$ is a cylinder with axes parallel to $v$ and base orthogonal to $u.$
\end{cor}

We remark, that to deduce this corollary, we used that homogeneity of (\ref{lemma2-sect3}) in $u$ and $v$, and noticed that it suffices to prove this fact only for unit vectors.

\medskip
\medskip

\textbf{Proof of Theorem \ref{thm-int}.} By the invariance of (\ref{LogBM-infin}) under the change $f\rightarrow f+t\langle x,n_x\rangle,$ it suffices to show that for any symmetric convex body $K$ and any vector $v\in\R^n$, letting $M=[-v,v]$, we have
$$n(n-1)V_{2}(K,M)+\int_{\partial K} \frac{h_M^2(n_x)}{\langle x,n_x\rangle}\leq \frac{n^2V_{1}(K,M)^2}{|K|}.
$$
Note that the function
$$|K+t[-v,v]|_n=|K|_n+2t|v|\cdot |K|v^{\perp}|_{n-1}$$
is linear in $t$, and therefore
$$V_2(K,[-v,v])=\frac{1}{n(n-1)}|K+t[-v,v]|''_{t}=0.$$
The statement therefore follows from Corollary \ref{point}, by integrating (\ref{lemma2-sect3}) on $\sfe$ with respect to the measure $|\langle u,v\rangle| dS_K(u)$, and in view of the fact that
$$nV_1(K,[-v,v])=\int_{\sfe} |\langle \theta,v\rangle|dS_K(\theta)=2|K|v^{\perp}|,$$
where in the last passage we used Cauchy's projection formula (\ref{Cauchy-formula}). 

Next, suppose equality holds. Then equality must hold in Corollary \ref{point} for all $u$ in the support of $dS_K$. Note that when $K$ is a cylinder with axes parallel to $v$ and base orthogonal to some vector $u_0,$ then all the vectors $u\in supp(S_K)\setminus \{u_0\}$ have the property $\langle u,v\rangle=0.$ We apply the equality case characterization from Corollary \ref{point} to every $u$ in the support of $|\langle \theta, v\rangle|dS_K(\theta)$, and conclude that, firstly, there is only one vector $u_0$ in the support of $dS_K$ for which $|\langle u_0,v\rangle|>0,$ and secondly, $K$ is a cylinder with axes parallel to $v$ and base orthogonal to $u_0.$ 


\medskip
\medskip

\textbf{Proof of Theorem \ref{thm-zono}.} Recall that any semi-norm $\|\cdot\|$ on $\R^n$ can be written as
$$\|u\|=\sup_{v\in\Omega} |\langle u,v\rangle|,$$
for some set $\Omega$. The conclusion thus follows by integrating (\ref{lemma1-sect3}) from Lemma \ref{max} on $\sfe$ with respect to the measure $\sup_{v\in\Omega}|\langle u,v\rangle| dS_K(u)$. 

Suppose now that the equality holds. Then the equality must hold in Lemma \ref{max}. Since $K$ is a bounded convex set with non-empty interior, by Claim \ref{claim} there is at least one vector $u_0$ in the support of $dS_K$ which is not orthogonal to all the vectors in $\Omega$ (unless $\Omega=\{0\}$, which is a trivial case anyway). Thus in order for the equality to hold in Theorem \ref{thm-zono}, $\Omega$ must consist only of vectors parallel to one another, by the equality case characterization in Lemma \ref{max}, applied with $u_0$. The remaining part of the equality case characterization follows from the equality case characterization in Theorem \ref{thm-int}. $\square$

\section{Proof of Theorem \ref{thm-bound}.}

We begin by formulating a Lemma, which follows from a Bochner-type identity obtained by Kolesnikov and Milman \cite{KM1}, which is a generalization of a classical result of R.C.~Reilly, along with the non-negativity of mixed volumes. Recall that for a matrix $A=(a_{ij})$, its Hilbert-Schmidt norm $\|A\|_{HS}=\sqrt{\sum_{i,j} a^2_{ij}}$.

\begin{lemma}\label{lemma-boch}
Let $K$ be $C^2$-smooth strictly convex body in $\R^n$. Let $\|\cdot\|$ be an arbitrary semi-norm in $\R^n$. Let $u:K\rightarrow\R$ be any $C^2$ function such that $\langle \nabla u,n_x\rangle=\|n_x\|$ for all $x\in\partial K.$ Then
$$\int_{K} \|\nabla^2 u\|_{HS}^2\leq \int_K (\Delta u)^2.$$
\end{lemma}
\begin{proof} Kolesnikov and Milman \cite{KM1} showed, for all $C^2$-smooth functions $u$ on $K,$ for which the integrals below exist, that:
\begin{align}
\label{railey}
\int_{K} (\Delta u)^2 dx & =\int_{K} ||\nabla^2 u||_{HS}^2dx+
\\&  \nonumber\int_{\partial K}  H_x \langle \nabla u,n_x\rangle^2 -2\langle \nabla_{\partial K} u, \nabla_{\partial K}  \langle \nabla u,n_x\rangle\rangle +\langle \mbox{\rm{II}} \nabla_{\partial K} u, \nabla_{\partial K}  u\rangle.
\end{align}
Recall that for any positive definite $n\times n$ matrix $A$ and for any $x,y\in\R^n$ we have
\begin{equation}\label{matrix}
\langle Ax,x\rangle+\langle A^{-1}y,y\rangle\geq 2 \langle x,y\rangle.
\end{equation}
As $K$ is convex, its second fundamental form $\rm{II}$ is positive definite, and consequently, letting 
$$f(x)=\langle \nabla u,n_x\rangle=\|n_x\|,$$ 
we get
\begin{equation}\label{sqf}
-2\langle \nabla_{\partial K} u, \nabla_{\partial K}  f\rangle +\langle \mbox{\rm{II}} \nabla_{\partial K} u, \nabla_{\partial K}  u\rangle\geq \langle \mbox{\rm{II}}^{-1}\nabla_{\partial K}  f,  \nabla_{\partial K}  f\rangle.
\end{equation}
By (\ref{railey}) and (\ref{sqf}), we have 
\begin{equation}\label{1!}
\int_{\partial K} H_x f^2-\langle \mbox{\rm{II}}^{-1}\nabla_{\partial K}  f,  \nabla_{\partial K}  f\rangle \leq \int_K (\Delta u)^2-\|\nabla^2 u\|_{HS}^2.
\end{equation}
Recall that for any semi-norm $\|\cdot\|$ there exists symmetric convex set $M$ (possibly with an empty interior) such that $\|\cdot\|_{M^{\circ}}=\|\cdot\|$. When $f(x)=\|n_x\|_{M^o},$ for some symmetric bounded convex set $M\subset\R^n$, we have
\begin{equation}\label{2!}
\int_{\partial K} H_x f^2-\langle \mbox{\rm{II}}^{-1}\nabla_{\partial K}  f,  \nabla_{\partial K}  f\rangle=|K+tM|''_0=n(n-1)V_2(K,M)\geq 0,
\end{equation}
and the desired statement follows from (\ref{1!}) and (\ref{2!}). See e.g. \cite{Col1}, \cite{KolMil-1}, as well as \cite{KM1}, \cite{KM2}, \cite{KolMil}, \cite{KolMilsupernew}, \cite{KolLiv}, for the proof of the first passage in (\ref{2!}), and e.g. Schneider \cite{book4} for the non-negativity of mixed volumes.
\end{proof}


We are now ready to prove Theorem \ref{thm-bound}. Fix a symmetric convex set $K$ in $\R^n$ with non-empty interior and fix a semi-norm $\|\cdot\|$ on $\R^n.$ Let $r$ be the inradius of $K$ (that is, the radius of the largest ball contained in $K$), and let $C_{poin}$ be the Poincare constant of $K$ (defined in the introduction (\ref{PCdef})).

Without loss of generality we may assume that $\|\cdot\|$ is in fact a norm, and is infinitely smooth. We may also assume that $K$ is strictly convex and the boundary of $K$ is of class $C^{\infty}.$ The general result would then follow by approximation.

Let $u: K\rightarrow\R$ be the solution of the Laplace equation with Neumann boundary condition
$$\langle \nabla u,n_x\rangle=\|n_x\|,\,\,\,\,x\in\partial K,$$
and 
$$\Delta u=\frac{\int_{\partial K} \|n_x\| }{|K|},\,\,\,\,x\in K.$$
It exists by the standard results from PDE, see, e.g., Evans \cite{evans}, and the function $|\nabla u|$ is continuously differentiable up to the boundary, under our regularity assumptions. We estimate
$$\int_{\partial K} \frac{\|n_x\|^2}{\langle x,n_x\rangle}\leq \frac{1}{r}\int_{\partial K}|\nabla u|\langle \nabla u, n_x\rangle,$$
in view of the fact that $\langle\nabla u, n_x\rangle=\|n_x\|\geq 0$ (see \cite{HKL} for a similar estimate). Note that, for any $\alpha, \beta>0,$
$$div(|\nabla u|\nabla u)=\Delta u |\nabla u|+\langle \nabla^2 u \frac{\nabla u}{|\nabla u|},\nabla u\rangle\leq \frac{\alpha}{2} (\Delta u)^2+\frac{1}{2\alpha} |\nabla u|^2+\frac{\beta}{2}\|\nabla^2 u\|_{HS}^2+\frac{1}{2\beta}|\nabla u|^2.$$
Thus, by divergence theorem, we get
$$\int_{\partial K} \frac{\|n_x\|^2}{\langle x,n_x\rangle}\leq \frac{1}{r}\int_K \frac{\alpha}{2} (\Delta u)^2+\frac{1}{2\alpha} |\nabla u|^2+\frac{\beta}{2}\|\nabla^2 u\|_{HS}^2+\frac{1}{2\beta}|\nabla u|^2 dx\leq$$
$$\frac{1}{r}\int_K \frac{\alpha}{2} (\Delta u)^2+\left(\frac{C^2_{poin}}{2\alpha}+\frac{\beta}{2}+\frac{C^2_{poin}}{2\beta}\right)\|\nabla^2 u\|_{HS}^2,$$
where in the last line we used the Poincare inequality coordinate-wise for $\nabla u,$ in view of the fact that $u$ is even and thus $\int_K \nabla u=0$. Lastly, we let $\alpha=\beta=C_{poin}$, and use Lemma \ref{lemma-boch}, which states 
$$\int_K \|\nabla^2 u\|_{HS}^2\leq \int_K(\Delta u)^2,$$ 
in order to conclude
$$\int_{\partial K} \frac{\|n_x\|^2}{\langle x,n_x\rangle}\leq\frac{2 C_{poin}}{r}\cdot \int_{K}(\Delta u)^2.$$
It remains to recall that $\Delta u$ is a constant function, and thus 
$$\int_{K}(\Delta u)^2dx=\frac{(\int_K \Delta u)^2}{|K|}=\frac{\left(\int_{\partial K} \|n_x\|\right)^2}{|K|},$$
where in the last passage we used the Divergence Theorem. We conclude that 
\begin{equation}\label{concl-sect5}
\int_{\partial K} \frac{\|n_x\|^2}{\langle x,n_x\rangle}\leq\frac{2 C_{poin}}{r} \cdot\frac{\left(\int_{\partial K} \|n_x\|\right)^2}{|K|}.
\end{equation}
Lastly, we note that for any volume-preserving linear transformation $T,$ and any symmetric bounded convex set $M$
$$\int_{\partial (TK)} \frac{\|n_x\|_M^2}{\langle x,n_x\rangle}=\int_{\sfe} \frac{\|u\|^2_M}{h_{TK}(u)} dS_{TK}(u)=\int_{\sfe} \frac{\|u\|^2_{(T^*)^{-1}M}}{h_{K}(u)} dS_{K}(u)=\int_{\partial K} \frac{\|n_x\|_{\bar{M}}^2}{\langle x,n_x\rangle},$$
where $\bar{M}=(T^*)^{-1}M.$ Similarly,
$$\int_{\partial (TK)} \|n_x\|_M=\int_{\partial K} \|n_x\|_{\bar{M}}.$$
See formula (2.3.2) in Artstein-Avidan, Giannopoulos, V. Milman \cite{AGM} and the discussion there for more details about this coordinate change. Therefore, (\ref{concl-sect5}) yields the statement of the theorem. $\square$

\section{Proof of the Proposition \ref{prop}.}

For symmetric convex bodies $K$ and $M$, we are looking to show that there exists a linear operator $T$ such that 
$$n(n-1)V_2(K,TM)+\int_{\sfe} \frac{h_{TM}^2(n_x)}{h_K}dS_K\leq \left(1+\frac{C\log n}{\sqrt{n}}\right)\cdot\frac{n^2 V_1(K,TM)^2}{|K|}.$$

Note that the statement is invariant under the operation of dilating $M$. Let $T$ be such an operator that brings $M$ to the position of minimal volume ratio with $K$. That is, suppose that $TM\subset K$ and $\frac{|K|}{|TM|}$ is minimal (among all the choices for $T\in GL_n$). Giannopoulos and Hartzoulaki \cite{GH} showed that in this case,
\begin{equation}\label{GH}
\left(\frac{|K|}{|TM|}\right)^{\frac{1}{n}}\leq C\sqrt{n}\log n.
\end{equation}
As $TM\subset K,$ we have $h_{TM}\leq h_K,$ and therefore
\begin{equation}\label{est1}
\int_{\sfe} \frac{h_{TM}^2}{h_K}dS_K\leq \int_{\sfe} h_{TM}dS_K.
\end{equation}
On the other hand, by Minkowski's first inequality,
\begin{equation}\label{est2}
\frac{n^2 V_1(K,TM)^2}{|K|}=\frac{nV_1(K,TM)}{|K|}\int_{\sfe} h_{TM}dS_K\geq n\left(\frac{|TM|}{|K|}\right)^{\frac{1}{n}}\int_{\sfe} h_{TM}dS_K.
\end{equation}
Combining (\ref{est1}) and (\ref{est2}) with the Giannopoulos-Hartzoulaki bound (\ref{GH}), we get
\begin{equation}\label{est3}
\int_{\sfe} \frac{h_{TM}^2}{h_K}dS_K\leq \frac{C\log n}{\sqrt{n}}\cdot \frac{n^2 V_1(K,TM)^2}{|K|}. 
\end{equation}

The above, together with Minkowski's Second inequality (\ref{MinkQ}) yields the desired result. $\square$

\section{Some questions and reductions.}

An interesting partial case of the local version of the Log-Brunn-Minkowski conjecture 
$$n(n-1)V_{2}(K,M)+\int_{\partial K} \frac{h_M^2(n_x)}{\langle x,n_x\rangle}\leq \frac{n^2V_{1}(K,M)^2}{|K|},$$
arises when $M=[-e_1,e_1]\times[-e_2,e_2]$, that is a two-dimensional square. 

\begin{proposition}\label{square}
Suppose the Log-Brunn-Minkowski conjecture holds. Then, for every symmetric convex body $K$,
\begin{equation}\label{eq-square}
8|K|span(e_1,e_2)^{\perp}|_{n-2}+\int_{\sfe} \frac{(|u_1|+|u_2|)^2}{h_K(u)}dS_K(u)\leq \frac{4\left(|K|e_1^{\perp}|_{n-1}+|K|e_2^{\perp}|_{n-1}\right)^2}{|K|_n}.
\end{equation}
\end{proposition}
\begin{proof} Let $M=[-e_1,e_1]\times[-e_2,e_2]$, which in fact also means $M=[-e_1,e_1]+[-e_2,e_2]$. By Cauchy's projection formula (\ref{Cauchy-formula}), and the linearity of mixed volumes,
$$nV_1(K,M)=nV_1(K,[-e_1,e_1])+nV_1(K,[-e_2,e_2])=2\left(|K|e_1^{\perp}|_{n-1}+|K|e_2^{\perp}|_{n-1}\right).$$
In view of the above, as well as the fact that $h_M(u)=|u_1|+|u_2|$, it remains to show that
\begin{equation}\label{airplane}
n(n-1)V_{2}(K,M)=8|K|span(e_1,e_2)^{\perp}|_{n-2}.
\end{equation}
Then we would get (\ref{LogBM-norms}), which itself is equivalent to the Log-Brunn-Minkowski inequality. 

We note that (\ref{airplane}) is well-known to experts: it is a special case of Theorem 5.3.1 from Schneider \cite{book4}. However, we present the proof for completeness. Without loss of generality, assume that $K$ has smooth boundary. We note that
\begin{equation}\label{statement}
n(n-1)V_{2}(K,M)=2\int_{Le_1} \frac{\langle \rm{II}e_1, e_2\rangle}{|\rm{II}e_1|}|\langle n_x,e_2\rangle|,
\end{equation}
where 
$$Le_1=\{x\in\partial K:\,\langle n_x,e_1\rangle=0\}.$$
Indeed, integrating by parts in $\partial K$, we get
$$\int_{\partial K} H_x |\langle n_x,e_1\rangle| |\langle n_x,e_2\rangle|  =-\int_{\partial K} div_{\partial K} (e_1-\langle n_x,e_1\rangle n_x)sign(\langle n_x,e_1\rangle) |\langle n_x,e_2\rangle|=$$
$$2\int_{Le_1} \frac{\langle \rm{II}e_1, e_2\rangle}{|\rm{II}e_1|}|\langle n_x,e_2\rangle|+\int_{\partial K} sign(\langle n_x,e_1\rangle)sign(\langle n_x,e_2\rangle)\langle \rm{II} (e_1-\langle n_x,e_1\rangle n_x),e_2-\langle n_x,e_2\rangle n_x\rangle,$$
and (\ref{statement}) follows by polarization.

Note that the projection of $Le_1$ onto $e_1^{\perp}$ is precisely the boundary of $K|e_1^{\perp}$. Furthermore, note that for $x\in Le_1$, the vector $\rm{II}e_1$ is orthogonal to the surface of $Le_1$. We conclude that
$$\int_{Le_1} \frac{\langle \rm{II}e_1, e_1\rangle}{|\rm{II}e_1|}|\langle n_x,e_2\rangle|=\int_{\partial (K|e_1^{\perp})} |\langle n_x,e_2\rangle|.$$

Applying the Cauchy projection formula (\ref{Cauchy-formula}) to $L=K|e_1^{\perp}$ we get 
$$\int_{\partial (K|e_1^{\perp})} |\langle n_x,e_2\rangle|=2|(K|e_1^{\perp})|e_2^{\perp}|_{n-2}=2|K|span(e_1,e_2)^{\perp}|_{n-2},$$
which yields (\ref{airplane}), and thus the Proposition is proven.
\end{proof}

\begin{remark} In some sense, the inequality (\ref{eq-square}) reminds of Bonnesen's inequality in two dimensions, which played an important role in the proof of the Log-Brunn-Minkowski conjecture in dimension two by B\"or\"oczky, Lutwak, Yang, Zhang \cite{BLYZ}. Additionally, (\ref{eq-square}) may be viewed as a ``localization in two directions'' of the Local Log-Brunn-Minkowski inequality. It is tempting to apply a term-by-term estimate, and to use Theorem \ref{thm-int} in order to deduce (\ref{eq-square}), however, unfortunately, the remaining inequality may be false. This can be seen from letting $K$ to be a hexagon on the plane which is close in its shape to a square. 
\end{remark}

By linearity of mixed volumes, we also note the following

\begin{proposition}\label{reduction}
Suppose for any pair of vectors $v$ and $w,$ the inequality 
$$4|K|span(v,w)^{\perp}|_{n-2}+\int_{\sfe} \frac{\langle u,v\rangle \langle u,w\rangle}{h_K(u)}dS_K(u)\leq \frac{2|K|v^{\perp}|_{n-1}\cdot |K|w^{\perp}|_{n-1}}{|K|_n}$$
is true for some symmetric convex body $K$. Then for any zonoid $M,$
\begin{equation}\label{Logbm}
n(n-1)V_{2}(K,M)+\int_{\partial K} \frac{h_M^2(n_x)}{\langle x,n_x\rangle}\leq \frac{n^2V_{1}(K,M)^2}{|K|}.
\end{equation}
\end{proposition}
\begin{proof} 

Suppose $Z=\sum_{i=1}^N [-v_i,v_i]$ is a zonotope. 
Note, by the linearity of mixed volumes, that (\ref{Logbm}) for $K$ and $M=Z$ rewrites as
$$\sum_{i,j} n(n-1)V(K,...,K, [-v_i,v_i], [-v_j,v_j])+\sum_{i,j} \int_{\sfe} \frac{\langle u,v_i\rangle \langle u,v_j\rangle}{h_K}dS_K\leq $$$$ \sum_{i,j} \frac{n^2 V_1(K,[-v_i,v_i])\cdot V_1(K,[-v_j,v_j])}{|K|}.$$
Recall that $n V_1(K,[-v_i,v_i])=2|K|v_i|_{n-1}$. Note also, letting $S_{ij}=[-v_i,v_i]\times [-v_j,v_j]$: 
$$n(n-1)V(K,...,K, [-v_i,v_i], [-v_j,v_j])=\frac{n(n-1)}{2} V_2(K,S_{ij})=4|K|span(v_i,v_j)^{\perp}|_{n-2},$$
where the first passage follows from the fact that $V_2(K, [-v,v])=0$ for any vector $v,$ and the second passage follows from (\ref{airplane}). Thus the desired inequality follows from the term-by-term application of our assumption when $i\neq j$ and Theorem \ref{thm-int} when $i=j$. The conclusion follows by approximation. 


\end{proof}

\end{document}